\documentclass[12pt]{amsart}
\usepackage{amssymb,verbatim,amscd,amsmath,graphicx,enumerate}
\usepackage{graphicx}
\usepackage{caption}
\usepackage{subcaption}
\usepackage{pdfsync}
\usepackage{fullpage}
\usepackage{color}
\usepackage{marginnote}
\usepackage{tikz}
\usepackage{placeins}%to put figures not past a certain point using command \FloatBarrier

\usetikzlibrary{patterns, arrows, positioning}
\usetikzlibrary{calc}
\usetikzlibrary{matrix}
\usepgflibrary{arrows}
\usetikzlibrary{arrows}
\usetikzlibrary{decorations.pathreplacing}
\usetikzlibrary{decorations.pathmorphing}
\usepackage{hyperref}
\def\qedbox{\hbox{\vbox{\hrule\hbox{\vrule\kern3pt\vbox{\kern6pt}\kern3pt\vrule}\hrule}}}
\def\eqed{\eqno\hbox{\quad\qedbox}}

\numberwithin{equation}{section}
\newtheorem{theorem}{Theorem}[section]
\newtheorem{lemma}[theorem]{Lemma}
\newtheorem{proposition}[theorem]{Proposition}
\newtheorem{corollary}[theorem]{Corollary}

\newtheorem{setting}[theorem]{Setting and Notation}

\theoremstyle{definition}
\newtheorem{definition}[theorem]{Definition}

\newtheorem{def-prop}[theorem]{Definition-Proposition}
\newtheorem{remark}[theorem]{Remark}
\newtheorem{example}[theorem]{Example}

\newtheorem*{acknowledgement}{Acknowledgements}

\newtheorem*{Mysketch}{Sketch of proof} % \newtheorem establishes the object heading
  {\pushQED{\qed}\begin{Mysketch}}
  {\popQED\end{Mysketch}}

\DeclareMathOperator{\inid}{in}
\DeclareMathOperator{\ini}{in_{\tau}}

\def\m{{\mathfrak m}}

\newcommand{\init}{{\rm{in}_{\tau'}}}

\def\lt#1{\mathop{\rm in}_{\tau}\nolimits(#1)}

\def\Rini{\hbox{\ensuremath{{\mathcal R}(\lt I)}}}
\def\Ri{\hbox{\ensuremath{{\mathcal R}(I)}}}

\def\Fini{\hbox{\ensuremath{{\mathcal F}(\lt I)}}}
\def\Fi{\hbox{\ensuremath{{\mathcal F}(I)}}}

\def\Sym{\hbox{\ensuremath{{\mathcal S}(I)}}}

\def\Symin{\hbox{\ensuremath{{\mathcal S}(\ini(I))}}}
\begin{document}

\title[Rees algebras of sparse determinantal ideals]{\bf Rees algebras of sparse determinantal ideals }

\author[Ela Celikbas]{Ela Celikbas}
\address{Ela Celikbas \\ Department of Mathematics\\
 West Virginia University\\
 Morgantown, W.V. 26506}
\email{ela.celikbas@math.wvu.edu}
\urladdr{https://math.wvu.edu/~ec0029}

\author[Emilie Dufresne]{Emilie Dufresne}
\address{Emilie Dufresne \\ Department of Mathematics\\
University of York\\
York, UK}
\email{emilie.dufresne@york.ac.uk}
\urladdr{https://www.york.ac.uk/maths/staff/emilie-dufresne/}

\author[Louiza Fouli]{Louiza Fouli}
\address{Louiza Fouli \\ Department of Mathematical Sciences \\
New Mexico State University\\
Las Cruces, NM 88003}
\email{lfouli@nmsu.edu}
\urladdr{http://www.web.nmsu.edu/~lfouli}

\author[Elisa Gorla]{Elisa Gorla}
\address{Elisa Gorla \\ Institut de Math\'ematiques\\
Universit\'e de Neuch\^atel\\
Rue Emile-Argand 11\\
CH-2000 Neuch\^atel\\
Switzerland}
\email{elisa.gorla@unine.ch}
\urladdr{http://members.unine.ch/elisa.gorla/}

\author[Kuei-Nuan Lin]{Kuei-Nuan Lin}
\address{Kuei-Nuan Lin\\
Department of Mathematics\\
Penn State University, 
Greater Allegheny campus\\
McKeesport, PA 15132}
\email{kul20@psu.edu}
\urladdr{https://sites.psu.edu/kul20}

\author[Claudia Polini]{Claudia Polini}
\address{Claudia Polini\\Department of Mathematics\\University of Notre Dame\\255 Hurley\\ Notre Dame, IN 46556 }
\email{cpolini@nd.edu}
\urladdr{https://www3.nd.edu/~cpolini}

\author[Irena Swanson]{Irena Swanson}
\address{Irena Swanson \\ Department of Mathematics\\
 Purdue University, 150 N. University Street, West Lafayette, IN 47907}
\email{irena@purdue.edu}
\urladdr{https://www.math.purdue.edu/~iswanso}

\thanks{Claudia Polini  was partially supported by NSF grant DMS-1902033}

\maketitle

\begin{abstract}
We determine the defining equations of the Rees algebra and of the special fiber ring of the ideal of maximal minors of a $2\times n$ sparse matrix. We prove that their initial algebras are ladder determinantal rings. This allows us to show that the Rees algebra and the special fiber ring are Cohen-Macaulay domains, they are Koszul, they have rational singularities in characteristic zero and are F-rational in positive characteristic.
\end{abstract}

\section{Introduction}\label{introduction}

Given an ideal $I$ in a Noetherian ring $R$, one can associate an algebra to $I$ known as the Rees algebra $\Ri$ of $I$. This algebra $\Ri=\bigoplus_{i\ge 0}I^i t^i$ is a subalgebra of $R[t]$, where $t$ is an indeterminate. It was introduced by Rees in 1956 in order to prove what is now known as the Artin-Rees Lemma \cite{Rees}.
Geometrically, the Rees algebra  corresponds to the blowup of ${\rm Spec}(R)$ along $V(I)$. If $R$ is local with maximal ideal $\m$ or graded with homogeneous maximal ideal $\m$, the special fiber ring of $I$ is the algebra $\Fi=\Ri\otimes R/\m$. This algebra corresponds to the special fiber of the blowup of ${\rm Spec}(R)$ along $V(I)$. Besides its connections to resolution of singularities, the study of Rees algebras plays an important role in many other active areas of research including multiplicity theory, equisingularity theory, asymptotic properties of ideals, and integral dependence.

Although blowing up is a fundamental operation in the study of birational varieties,
an explicit understanding of this process remains an open problem. In particular, a key objective in this area is to express the Rees algebra and the special fiber ring as quotients of a polynomial ring, henceforth to determine their defining ideals. 

This question is wide open even for the simplest classes of ideals, including ideals generated by forms of the same degree in a polynomial ring. These are precisely the ideals generated by forms parametrizing a variety in projective space. The implicit equations of these varieties can be obtained from the defining ideal of the Rees ring. Indeed, the bihomogenous coordinate ring of the graph of the morphism defined by the forms is the Rees algebra of the ideal $I$. The homogeneous coordinate ring of the variety parametrized by the forms is the special fiber ring.

As the graph of a map carries more information than its image, even a partial understanding of the Rees ring such as the bigraded degrees of its defining equations, the Betti numbers, or the regularity of the defining ideal can be instrumental to the study of the variety. Determining the defining equations of the Rees algebra is a difficult problem in elimination theory, studied by commutative algebraists, algebraic geometers, and applied mathematicians in geometric modeling, see e.g. \cite{Buse, BJ,CWL, Cox, SCG}. Answers to these questions also have applications to the study of chemical reaction networks~\cite{CLS}.

The goal of this paper is to determine the defining equations of the Rees algebra and of the special fiber ring of the ideals generated by the maximal minors of sparse $2 \times n$ matrices. Sparse matrices are matrices whose entries are either zeroes or distinct variables. Their degeneracy loci were first studied by Giusti and Merle in the 80's. In~\cite{GM} they compute the codimension of their defining ideals and characterize when these ideals are prime or Cohen-Macaulay. Boocher in 2012 proved in~\cite{Boocher} that a minimal free resolution of the ideals of maximal minors of sparse matrices can be obtained from the Eagon-Northcott complex via a pruning method. In the same paper, he shows that the natural generators form a universal Gr\"obner basis.

In the case of a generic matrix, that is a matrix whose entries are distinct variables, the special fiber ring of the ideal of maximal minors is the coordinate ring of a Grassmannian variety. The fact that the Pl\"{u}cker relations define the Grassmannian variety is a classical theorem, see e.g.~\cite{BV}. The Rees algebra and the special fiber ring of ideals of maximal minors of generic matrices are Algebras with Straightening Laws (ASLs) in the sense of~\cite{DEP}, see~\cite{E, EHu}. Since the straightening relations come from the Pl\"{u}cker relations 
and the defining equations of the symmetric algebra, it follows that $\Ri$ is of fiber type, see \cite{DEP, E}, \cite[Lemma~2.2.1]{BST}, \cite[Proposition~4.2]{BV}.

In addition, as the posets defining $\Ri$ and $\Fi$ are wonderful in the sense of~\cite{EHu}, it follows that both algebras are Cohen-Macaulay, see \cite[Proposition~2.6]{EHu}. The normality of $\Fi$ follows immediately from the Cohen-Macaulay property since the Grassmannian variety is smooth.  Trung in \cite{Trung} proved that the powers and symbolic powers of $I$ coincide and therefore $\Ri$ is normal. From the deformation theorem developed in \cite{CHV}, see for instance \cite[Proposition~3.6]{BCV}, one can see that the Rees algebra of the ideal of maximal minors is defined by a Gr\"{o}bner basis of quadrics. The same statement holds for $\Fi$. 
Therefore, in the generic case both $\Ri$ and $\Fi$ are Koszul algebras and according to \cite{Bl} the ideal $I$ and all its powers have a linear resolution.

Our main result shows that the above properties of the blowup algebras of ideals of maximal minors of a generic matrix still hold in the case of sparse $2\times n$ matrices.

Now let $I$ be the ideal of maximal minors of a $2\times n$ sparse matrix. Inspired by the pioneering work of Conca, Herzog, and Valla~\cite{CHV}, we study the initial algebra of the Rees algebra of $I$. Our main technique is SAGBI bases~\cite{KM, RS}, an analogue for algebras of Gr\"obner bases for ideals. 
 
This approach was successfully used to study the Rees algebras of other families of ideals~\cite{ALL,CHV,LS2}.

First we prove that the initial algebra of the Rees algebra of $I$ is the Rees algebra of the initial ideal of $I$ with respect to a suitable order (see Theorem~\ref{initial rees thm}). Using deformation theory, we transfer properties from the Rees algebra of the initial ideal to the Rees algebra of the ideal itself. One advantage of this approach is that it allows us to reduce to the study of the Rees algebra of the initial ideal, which is not just a monomial algebra, but also the Rees algebra of the edge ideal of a graph and a ladder determinantal ring. These objects have been studied extensively and one can draw a plethora of information that allows us to describe these algebras in full detail, see among others~\cite{Conca, CN, SVV, V, Wang, Wang1}.

A key step in our proof that $\Rini$ is the initial algebra of $\Ri$ is Lemma~\ref{I^2=}, where we prove that taking the initial ideal commutes with powers. The main idea behind the proof is a comparison of the Hilbert functions of $I^2$ and $(\ini(I))^2$, an approach which was first used in~\cite{GMN}. We then use a lifting technique to obtain the defining equations of the Rees algebra and of the special fiber ring. Interestingly, they turn out to be the specialization of the defining equations of the Rees algebra and of the special fiber ring in the generic case.

The general question of understanding the Rees algebra and the special fiber ring of the ideal $I$ of maximal minors of a sparse $m\times n$ matrix is still open. In Remark~\ref{ASL} we propose a different approach, which applies to sparse matrices whose zero region has a special shape. These sparse matrices are exactly those that have the property that a maximal minor is non-zero if and only if the product of the elements on its diagonal is non-zero. This yields a nice combinatorial description of the initial ideal of $I$. Our arguments allow us to compute the equations of the special fiber ring and the Rees algebra and to establish algebraic properties such as normality, Cohen-Macaulayness, and Koszulness. We conjecture that these properties hold in general.

Our main results are summarized in the following.
\begin{theorem}\label{mainresults}
Let $X$ be a sparse $2 \times n$ matrix and $I = I_2(X)$.
\begin{enumerate}[$($a$)$]
\item The defining ideal of $\Fi$ is generated by the Pl\"ucker relations on the $2\times 2$-minors of $X$ and these form a Gr\"obner basis of the ideal they generate.
\item $\Ri$ is of fiber type, that is, its defining ideal is generated by the relations of the symmetric algebra of $I$ and by the Pl\"ucker relations on the $2\times 2$-minors of $X$. Moreover, these equations form a Gr\"obner basis of the ideal they generate.
\item $\Ri$ and $\Fi$ have rational singularities in characteristic zero and they are $F$-rational in positive characteristic. In particular, they are Cohen-Macaulay normal domains.
\item $\Ri$ and $\Fi$ are Koszul algebras. In particular, the powers of $I$ have a linear resolution.
\item The natural algebra generators of $\Ri$ and $\Fi$ are SAGBI bases of the algebras they generate.
\end{enumerate}
\end{theorem}

\section{Notation}\label{sectsetting}

This section is devoted to the setup and notation we will use throughout the paper. 
Let $K$ be a field, $n\ge 3$ a positive integer,
and $X=(x_{ij})$ a $2 \times n$ sparse matrix of indeterminates over $K$, i.e., a matrix whose entries $x_{ij}$ are either distinct indeterminates or zero. We use the notation $x_{i,j}$ when it is not clear what the two subscripts are. Let $R=K[X]$ denote the polynomial ring over $K$ in the variables appearing in $X$ and let $I$ be the ideal generated by the $2 \times 2$ minors of $X$.

Up to permuting the rows and the columns of $X$, we may assume that there exist $r$ and $s$ with $0\le n-r-s\le r < n$ such that 
$$
X= \begin{pmatrix}
x_{1,1} & \cdots & x_{1,r} & x_{1,r+1} & \cdots & x_{1,r+s} &  0 & \cdots & 0 \\
0 & \cdots & 0 & x_{2,r+1} & \cdots & x_{2,r+s} & x_{2,r+s+1} & \cdots & x_{2,n}
\\
\end{pmatrix}.
$$
If $r = 0$, then $n=s$ and $X$ is a generic matrix. Since the results obtained in this paper are known in this case, 
we may assume without loss of generality that $r\ge 1$. 
Let $f_{ij}$ denote the $2 \times 2$ minor of $X$ that involves columns $i$ and $j$. 
Then $$I=(f_{ij}\mid 1\le i<j\le n)\subseteq R.$$
Notice that $f_{ij} = - f_{ji}$ and that $f_{ij}$ is a monomial or zero, unless $r+1\le i, j\le r+s$.

By the form of $X$, if a minor is non-zero, then the product of the entries on its diagonal is also non-zero. This means that we can choose a diagonal term order on $K[X]$, that is, a term order with the property that the leading term of each minor of $X$ is the product of the elements on its diagonal. An example of a diagonal term order is the lexicographic order with $x_{1,1}>\ldots>x_{1,r+s}>x_{2,r+1}>\ldots>x_{2,n}$. The maximal minors of $X$ form a diagonal Gr\"obner basis for $I$ by \cite[Proposition~5.4]{Boocher}.

Throughout the paper, we fix a diagonal term order $\tau$. The minimal generating set for $\ini(I)$ with respect to $\tau$ is described next. 

\begin{proposition}[\cite{Boocher}, Proposition~5.4]\label{gb of ini}
Let $X$ be a sparse $2\times n$ matrix, $I$ the ideal of $2\times 2$ minors of $X$, and $\tau$ a diagonal term order.
The initial ideal of $I$ is $$\ini(I)=(x_{1i}x_{2j}\mid 1\le i\le r+s, \max\{r,i\}<j\le n).$$
\end{proposition}

It turns out that $\ini(I)$ corresponds to a {\it Ferrers diagram} as in Figure~\ref{fig:1}.

\begin{figure}[h]
	\centering
	\begin{tikzpicture}[ every node/.style={scale=0.9}]
	\node [label = above: {\small$x_{2,n}$}] at (5.3,5) {};
	\node [label = above: {\small $\cdots$}] at (6.5,5) {};
	\node [label = above: {\small$x_{2,r+s+1}$}] at (7.42,5) {};
	\node [label = above: {\small$x_{2,r+s}$}] at (8.5,5) {};
	\node [label = above: {\small$\cdots$}] at (9.5,5) {};
	\node [label = above: {\small$x_{2,r+1}$}] at (11,5) {};
	\node [label = left: { \small$x_{1,1}$}] at (5,4.8){};
	\node [label = left: {\small$\vdots$}] at (4.7,4){};
	\node [label = left: {\small$x_{1,r}$}] at (5,3.2){};
	\node [label = left: {\small$x_{1,r+1}$}] at (5,2.8){};
	\node [label = left: {\small$\vdots$}] at (4.7,1){};
	\node [label = left: {\small$x_{1,r+s}$}] at (5,-0.4){};
	\draw [line width=1pt, color=black] (5,5)--(11.5,5);
	\draw[line width=1pt] (5,5)--(5,-.5);
	\draw[line width=1pt](5,3)--(11.5,3);
	\draw[line width=1pt](8,5)--(8,-.5);
	\draw[line width=1pt](9,1)--(9,0.5);
	\draw[line width=1pt](8.5,0.5)--(8.5,0);
	\draw[line width=1pt](10,2)--(10,1.5);
	\draw[line width=1pt](9.5,1.5)--(9.5,1);
	\draw[line width=1pt](11,3)--(11,2.5);
	\draw[line width=1pt](10.5,2.5)--(10.5,2);
	\draw[line width=1pt](11.5,5)--(11.5,3);
	\draw[line width=1pt](10.5,2.5)--(11,2.5);
	\draw[line width=1pt](10,2)--(10.5,2);
	\draw[line width=1pt](9.5,1.5)--(10,1.5);
	\draw[line width=1pt](9,1)--(9.5,1);
	\draw[line width=1pt](8,0)--(8.5,0);
	\draw[line width=1pt](8.5,0.5)--(9,0.5);
	\draw[line width=1pt](5,-0.5)--(8,-.5);

	\node [label = left: { \Large $\bf{A}$}] at (7,4){};
	\node [label = left: { \Large $\bf{B}$}] at (9.5,4){};
	\node [label = left: { \Large $\bf{C}$}] at (7,2){};
	\node [label = left: { \Large $\bf{D}$}] at (9.5,2){};
	\end{tikzpicture}
	\caption{ Ferrers diagram }
	\label{fig:1}
\end{figure}
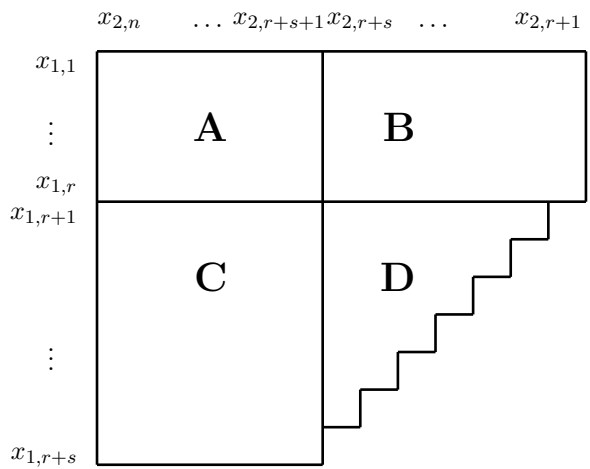

\begin{corollary} \label{cor init Ferrers}
The ideal $\ini(I)$ is the edge ideal of a Ferrers bipartite graph
whose vertex sets are
$$V=\{x_{1,1}, \ldots, x_{1,r+s}\}, \;\; W=\{x_{2,n}, \ldots ,x_{2,r+1}\},$$
and whose partition is $$\lambda=(\underbrace{n-r, \ldots, n-r}_{r \ {\rm{ times }}}, n-r-1, n-r-2, \ldots, n-r-s).$$ 
In other words, the first $r$ elements of $V$ have edges connecting them to all $n-r$ elements of $W$ Moreover, for $i>r$, the $i$th element of~$V$ is connected by an edge to the first $n-i$ elements of $W$.
\end{corollary}

We say that $\ini(I)$ is the {\it Ferrers ideal} $I_{\lambda}$.
See \cite{CN} for more on Ferrers graphs, diagrams, and ideals.

\begin{example} \label{running ex}
For the matrix
$$
X= \begin{pmatrix}
x_{11} & x_{12}& x_{13} & x_{14} & x_{15} & x_{16} &x_{17}&0&0\\
0 & 0 & 0 &x_{24} &x_{25}&x_{26}&x_{27} &x_{28} &x_{29}\\
\end{pmatrix}
$$
the ideal $\ini(I)$ is the Ferrers ideal $I_{\lambda}$
for the partition $\lambda=(6,6,6,5,4,3,2)$. Its Ferrers diagram is depicted in Figure~\ref{fig:2}.
\end{example}

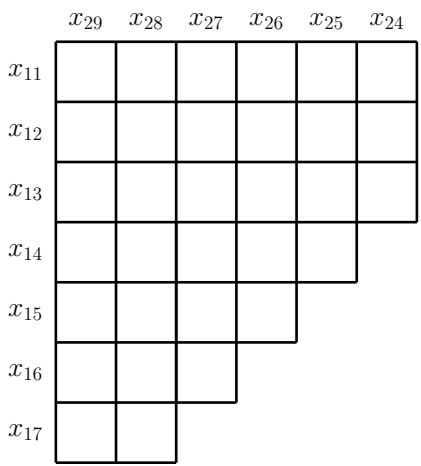
\begin{figure}[h]
	\centering
	\begin{tikzpicture}[scale=0.8, every node/.style={scale=0.6}]
	\node [label = above: {\Large $x_{29}$}] at (5.5,5) {};
	\node [label = above: {\Large $x_{28}$}] at (6.5,5) {};
	\node [label = above: {\Large $x_{27}$}] at (7.5,5) {};
	\node [label = above: {\Large $x_{26}$}] at (8.5,5) {};
	\node [label = above: {\Large $x_{25}$}] at (9.5,5) {};
	\node [label = above: {\Large $x_{24}$}] at (10.5,5) {};
	\node [label = left: {\Large $x_{11}$}] at (5,4.5){};
	\node [label = left: {\Large $x_{12}$}] at (5,3.5){};
	\node [label = left: {\Large $x_{13}$}] at (5,2.5){};
	\node [label = left: {\Large $x_{14}$}] at (5,1.5){};
	\node [label = left: {\Large $x_{15}$}] at (5,0.5){};
	\node [label = left: {\Large $x_{16}$}] at (5,-0.5){};
	\node [label = left: {\Large $x_{17}$}] at (5,-1.5){};
	\draw [line width=1pt, color=black] (5,5)--(11,5);
	\draw [line width=1pt] (5,5)--(5,-2);
	\draw [line width=1pt] (5,4)--(11,4);
	\draw [line width=1pt](6,5)--(6,-2);
	\draw [line width=1pt](5,3)--(11,3);
	\draw [line width=1pt](7,5)--(7,-2);
	\draw [line width=1pt](8,5)--(8,-1);
	\draw [line width=1pt](9,5)--(9,0);
	\draw [line width=1pt](10,5)--(10,1);
	\draw [line width=1pt](11,5)--(11,2);
	\draw [line width=1pt](7,2)--(7,-1);
	\draw [line width=1pt](5,2)--(11,2);
	\draw [line width=1pt](5,1)--(10,1);
	\draw [line width=1pt](5,0)--(9,0);
	\draw [line width=1pt](5,-1)--(8,-1);
	\draw [line width=1pt](5,-2)--(7,-2);
	\end{tikzpicture}
\caption{Ferrers diagram for $\lambda=(6,6,6,5,4,3,2)$} \label{fig:2}
\end{figure}

\begin{definition}[\cite{CN}]\label{L lambda}
The {\it one-sided ladder associated to} a Ferrers ideal $I_{\lambda}$
is the ladder $L_{\lambda}$ with the same shape as the Ferrers diagram for $\lambda$ and for which the entry in row $x_{1i}$ and column $x_{2j}$ is $y_{ij}$.

The {\it ladder associated to} $\ini(I)=I_{\lambda}$
is the two-sided ladder $L_{\lambda}'$
obtained from~$L_{\lambda}$ by adding a row of $n-r$ boxes at the top and a column of $r+s$ boxes on the left. The new boxes are filled with the variables from $W$ and $V$, respectively,
in their listed order.
\end{definition}

Illustrations of $L_{\lambda}$ and $L'_{\lambda}$
for $\lambda = (6,6,6,5,4,3,2)$ are in Figure~\ref{fig:3}.

\begin{figure}[h]
	\begin{tikzpicture}[scale=0.8, every node/.style={scale=0.6}]
	\node [label = above: {\Large $x_{29}$}] at (5.5,5) {};
	\node [label = above: {\Large $x_{28}$}] at (6.5,5) {};
	\node [label = above: {\Large $x_{27}$}] at (7.5,5) {};
	\node [label = above: {\Large $x_{26}$}] at (8.5,5) {};
	\node [label = above: {\Large $x_{25}$}] at (9.5,5) {};
	\node [label = above: {\Large $x_{24}$}] at (10.5,5) {};
	\node [label = left: {\Large $x_{11}$}] at (5,4.4){};
	\node [label = left: {\Large $x_{12}$}] at (5,3.4){};
	\node [label = left: {\Large $x_{13}$}] at (5,2.4){};
	\node [label = left: {\Large $x_{14}$}] at (5,1.4){};
	\node [label = left: {\Large $x_{15}$}] at (5,0.4){};
	\node [label = left: {\Large $x_{16}$}] at (5,-0.6){};
	\node [label = left: {\Large $x_{17}$}] at (5,-1.6){};
	\draw [line width=1pt, color=black] (5,5)--(11,5);
	\draw [line width=1pt] (5,5)--(5,-2);
	\draw [line width=1pt] (5,4)--(11,4);
	\draw [line width=1pt](6,5)--(6,-2);
	\draw [line width=1pt](5,3)--(11,3);
	\draw [line width=1pt](7,5)--(7,-2);
	\draw [line width=1pt](8,5)--(8,-1);
	\draw [line width=1pt](9,5)--(9,0);
	\draw [line width=1pt](10,5)--(10,1);
	\draw [line width=1pt](11,5)--(11,2);
	\draw [line width=1pt](7,2)--(7,-1);
	\draw [line width=1pt](5,2)--(11,2);
	\draw [line width=1pt](5,1)--(10,1);
	\draw [line width=1pt](5,0)--(9,0);
	\draw [line width=1pt](5,-1)--(8,-1);
	\draw [line width=1pt](5,-2)--(7,-2);
	\node [label = above: {\Large $y_{19}$}] at (5.5,4) {};
	\node [label = above: {\Large $y_{18}$}] at (6.5,4) {};
	\node [label = above: {\Large $y_{17}$}] at (7.5,4) {};
	\node [label = above: {\Large $y_{16}$}] at (8.5,4) {};
	\node [label = above: {\Large $y_{15}$}] at (9.5,4) {};
	\node [label = above: {\Large $y_{14}$}] at (10.5,4) {};
	\node [label = above: {\Large $y_{29}$}] at (5.5,3) {};	
	\node [label = above: {\Large $y_{28}$}] at (6.5,3) {};
	\node [label = above: {\Large $y_{27}$}] at (7.5,3) {};
	\node [label = above: {\Large $y_{26}$}] at (8.5,3) {};
	\node [label = above: {\Large $y_{25}$}] at (9.5,3) {};
	\node [label = above: {\Large $y_{24}$}] at (10.5,3) {};
	\node [label = above: {\Large $y_{39}$}] at (5.5,2) {};	
	\node [label = above: {\Large $y_{38}$}] at (6.5,2) {};
	\node [label = above: {\Large $y_{37}$}] at (7.5,2) {};
	\node [label = above: {\Large $y_{36}$}] at (8.5,2) {};
	\node [label = above: {\Large $y_{35}$}] at (9.5,2) {};
	\node [label = above: {\Large $y_{34}$}] at (10.5,2) {};
	\node [label = above: {\Large $y_{49}$}] at (5.5,1) {};	
	\node [label = above: {\Large $y_{48}$}] at (6.5,1) {};
	\node [label = above: {\Large $y_{47}$}] at (7.5,1) {};
	\node [label = above: {\Large $y_{46}$}] at (8.5,1) {};
	\node [label = above: {\Large $y_{45}$}] at (9.5,1) {};
	\node [label = above: {\Large $y_{59}$}] at (5.5,0) {};	
	\node [label = above: {\Large $y_{58}$}] at (6.5,0) {};
	\node [label = above: {\Large $y_{57}$}] at (7.5,0) {};
	\node [label = above: {\Large $y_{56}$}] at (8.5,0) {};
	\node [label = above: {\Large $y_{69}$}] at (5.5,-1) {};	
	\node [label = above: {\Large $y_{68}$}] at (6.5,-1) {};
	\node [label = above: {\Large $y_{67}$}] at (7.5,-1) {};
	\node [label = above: {\Large $y_{79}$}] at (5.5,-2) {};
	\node [label = above: {\Large $y_{78}$}] at (6.5,-2) {};
	\end{tikzpicture}
\hskip1cm
	\begin{tikzpicture}[scale=0.8, every node/.style={scale=0.6}]
	\node [label = above: {\Large $x_{29}$}] at (5.5,5) {};
	\node [label = above: {\Large $x_{28}$}] at (6.5,5) {};
	\node [label = above: {\Large $x_{27}$}] at (7.5,5) {};
	\node [label = above: {\Large $x_{26}$}] at (8.5,5) {};
	\node [label = above: {\Large $x_{25}$}] at (9.5,5) {};
	\node [label = above: {\Large $x_{24}$}] at (10.5,5) {};
	\node [label = left: {\Large $x_{11}$}] at (5,4.4){};
	\node [label = left: {\Large $x_{12}$}] at (5,3.4){};
	\node [label = left: {\Large $x_{13}$}] at (5,2.4){};
	\node [label = left: {\Large $x_{14}$}] at (5,1.4){};
	\node [label = left: {\Large $x_{15}$}] at (5,0.4){};
	\node [label = left: {\Large $x_{16}$}] at (5,-0.6){};
	\node [label = left: {\Large $x_{17}$}] at (5,-1.6){};
	\draw [line width=1pt, color=black] (4,5)--(11,5);
	\draw [line width=1pt] (4,5)--(4,-2);
	\draw [line width=1pt] (5,6)--(5,-2);
	\draw [line width=1pt] (4,4)--(11,4);
	\draw [line width=1pt] (5,6)--(11,6);
	\draw [line width=1pt](6,6)--(6,-2);
	\draw [line width=1pt](4,3)--(11,3);
	\draw [line width=1pt](7,6)--(7,-2);
	\draw [line width=1pt](8,6)--(8,-1);
	\draw [line width=1pt](9,6)--(9,0);
	\draw [line width=1pt](10,6)--(10,1);
	\draw [line width=1pt](11,6)--(11,2);
	\draw [line width=1pt](7,2)--(7,-1);
	\draw [line width=1pt](4,2)--(11,2);
	\draw [line width=1pt](4,1)--(10,1);
	\draw [line width=1pt](4,0)--(9,0);
	\draw [line width=1pt](4,-1)--(8,-1);
	\draw [line width=1pt](4,-2)--(7,-2);
	\node [label = above: {\Large $y_{19}$}] at (5.5,4) {};
	\node [label = above: {\Large $y_{18}$}] at (6.5,4) {};
	\node [label = above: {\Large $y_{17}$}] at (7.5,4) {};
	\node [label = above: {\Large $y_{16}$}] at (8.5,4) {};
	\node [label = above: {\Large $y_{15}$}] at (9.5,4) {};
	\node [label = above: {\Large $y_{14}$}] at (10.5,4) {};
	\node [label = above: {\Large $y_{29}$}] at (5.5,3) {};	
	\node [label = above: {\Large $y_{28}$}] at (6.5,3) {};
	\node [label = above: {\Large $y_{27}$}] at (7.5,3) {};
	\node [label = above: {\Large $y_{26}$}] at (8.5,3) {};
	\node [label = above: {\Large $y_{25}$}] at (9.5,3) {};
	\node [label = above: {\Large $y_{24}$}] at (10.5,3) {};
	\node [label = above: {\Large $y_{39}$}] at (5.5,2) {};	
	\node [label = above: {\Large $y_{38}$}] at (6.5,2) {};
	\node [label = above: {\Large $y_{37}$}] at (7.5,2) {};
	\node [label = above: {\Large $y_{36}$}] at (8.5,2) {};
	\node [label = above: {\Large $y_{35}$}] at (9.5,2) {};
	\node [label = above: {\Large $y_{34}$}] at (10.5,2) {};
	\node [label = above: {\Large $y_{49}$}] at (5.5,1) {};	
	\node [label = above: {\Large $y_{48}$}] at (6.5,1) {};
	\node [label = above: {\Large $y_{47}$}] at (7.5,1) {};
	\node [label = above: {\Large $y_{46}$}] at (8.5,1) {};
	\node [label = above: {\Large $y_{45}$}] at (9.5,1) {};
	\node [label = above: {\Large $y_{59}$}] at (5.5,0) {};	
	\node [label = above: {\Large $y_{58}$}] at (6.5,0) {};
	\node [label = above: {\Large $y_{57}$}] at (7.5,0) {};
	\node [label = above: {\Large $y_{56}$}] at (8.5,0) {};
	\node [label = above: {\Large $y_{69}$}] at (5.5,-1) {};	
	\node [label = above: {\Large $y_{68}$}] at (6.5,-1) {};
	\node [label = above: {\Large $y_{67}$}] at (7.5,-1) {};
	\node [label = above: {\Large $y_{79}$}] at (5.5,-2) {};
	\node [label = above: {\Large $y_{78}$}] at (6.5,-2) {};
	\end{tikzpicture}

\caption{Ladders $L_\lambda$ and $L_\lambda'$ for $\lambda=(6,6,6,5,4,3,2)$}
\label{fig:3}
\end{figure}

\FloatBarrier 

Notice that the $x_{ij}$ are only row and column markers in $L_\lambda$,
whereas in $L'_\lambda$ they are entries of the ladder.
The entries in $L_\lambda$ and in $L'_\lambda$
are distinct variables. Therefore, results for ladder determinantal ideals apply, see for instance \cite{Conca, Nar}.

\section{Rees algebras of $I$ and of its initial ideal}\label{sect Rees alg}

The Rees algebra of an ideal can be realized as a quotient of a polynomial ring. 
When $I$ is an ideal generated by $n$ elements, say $f_1, \ldots, f_n$, we let $y_1, \ldots, y_n$ be variables over $R$ and consider the map from $R[y_1, \ldots, y_n]$ to $\Ri$ that maps $y_i$ to $f_it$. Hence $\Ri\cong R[y_1, \ldots, y_n]/J$, where $J$ is the {\it defining ideal} of the algebra, and the {\it defining equations} are a system of generators of $J$.
The defining equations of the Rees algebra 
are in general difficult to compute or determine theoretically. The largest $y$-degree of a minimal generator of the defining ideal $J$ is the {\it relation type} of the ideal and plays an important role in the study of blowup algebras. Finally, we say that the Rees algebra of $I$ is of {\it fiber type} if the defining ideal of the Rees algebra is generated by the defining equations of the symmetric algebra and the defining equations of the special fiber ring.

Given a term order $\tau$ on a polynomial ring $R$ over a field $K$, one can extend it to a term order $\tau'$ on $R[t]$ as follows. Let $a,b\in R$ be monomials and let $i,j$ be non-negative integers. Define
\begin{equation}\label{tau'}
at^i<_{\tau'} bt^j \ \ \Leftrightarrow \ \ i<j \ \ \mbox{ or } \ \ i=j \mbox{ and } a<_{\tau}b.
\end{equation}
For a $K$-subalgebra $A$ of $R[t]$, one defines $\init(A)$ as the $K$-algebra generated by all initial monomials of elements in $A$. 
When $A$ is homogeneous in the variable $t$, then $\init(A)=\oplus_{i \ge 0} \init(A_i)$, where $A_i$ is the set of elements of $A$ that are homogeneous of degree $i$ in $t$. In particular, in our setting $$\init(\Ri)=\oplus_{i \ge 0} \init(\Ri_i)=\oplus_{i\ge 0} \ini(I^i)t^i.  $$ Since $\Rini=\oplus_{i \ge 1} (\ini(I))^it^i$, in order to prove that $\init(\Ri)=\Rini$, it suffices to show that $(\ini(I))^i=\ini(I^i)$ for all $i \ge 1$. A result by Conca, Herzog, and Valla states that it suffices to check the equality up to the relation type of $\ini(I)$.

\begin{theorem}\cite[Corollary~2.8]{CHV} \label{reltype thm}
Let $R$ be a polynomial ring, $I$ a homogeneous ideal in $R$, and $\tau$ a term order on $R$.  Suppose that $(\ini(I))^i=\ini(I^i)$ for $1\le i \le {\rm{reltype}}(\ini(I))$. Then
$$(\ini(I))^i=\ini(I^i) \mbox{ for all } i \ge 1 \ \ \mbox{ and } \ \ {\rm{reltype}}(I) \le {\rm{reltype}}(\ini(I)).$$
\end{theorem}

The equality $\init(\Ri)=\Rini$ was established by Conca in~\cite[Theorem~2.1]{ConcaGB} for the ideal $I$ of maximal minors of a generic $m\times n$ matrix. The following lemma is the key to establishing $\init(\Ri)=\Rini$ in our case.

\begin{lemma}\label{I^2=}
Let $X$ be a sparse $2\times n$ matrix, $I$ the ideal of $2\times 2$ minors of $X$, and $\tau$ a diagonal term order on $R=K[X]$. Then $\ini(I^2)=(\ini(I))^2$.
\end{lemma}

\begin{proof}
We proceed by induction on $n\ge 2$. For $n=2$, the ideal $I$ is principal and the result holds automatically. To proceed we choose a decomposition of $I$ as follows.

Let $I_1$ be the ideal generated by the variables $x_{2j}$ for $r+1 \le j \le n$ and let $I_2$ be the ideal generated by the $2 \times 2$ minors of the submatrix of $X$ obtained by deleting the first column of $X$. Since $r\ge 1$, then $I = x_{11} I_1 + I_2$. By Proposition~\ref{gb of ini}, $\ini(I) = x_{11} I_1 + \ini(I_2)$. Notice here that for $I_2$ we may have that its corresponding $r$ is $0$. In that case $\ini(I_2^2)=(\ini(I_2))^2$ by \cite[Theorem~2.1]{ConcaGB}. Otherwise, the equality holds by induction.

Certainly $(\ini(I))^2 \subseteq \ini(I^2)$. To prove equality it is enough to show that
the Hilbert functions of $(\ini(I))^2$ and $I^2$ are the same, since the Hilbert function of $I^2$ is the same as the Hilbert function of $\ini(I^2)$.

By induction on $n$, the Hilbert functions of $I_2^2$ and $(\ini(I_2))^2$ are the same. We have $$I^2=(x_{11}^2I_1^2+x_{11}I_1I_2)+I_2^2 \quad \mbox{and }\quad (\ini(I))^2=(x_{11}^2I_1^2+x_{11}I_{1}\ini(I_2))+(\ini(I_2))^2.$$
Both of these ideals are of the form
$J^2 = (x_{11}^2 I_1^2 + x_{11}I_1J_2) + J_2^2$,
with $J_2 = I_2$ and $J = I$ in the former case
and with $J_2 = \ini(I_2)$ and $J = \ini(I)$ in the latter case.
Each case gives rise to the short exact sequence:
\begin{eqnarray} \label{ses}
0 \longrightarrow (x_{11}^2I_1^2+x_{11}I_1J_2)\cap J_2^2\longrightarrow
(x_{11}^2I_1^2+x_{11}I_1J_2)\oplus J_2^2\longrightarrow J^2 \longrightarrow
0.
\end{eqnarray}
Since the variable $x_{11}$ is a non-zerodivisor on $J_2^2$,
and since $J_2 \subseteq I_1$,
the intersection on the left-hand side of the sequence is:
\begin{align*}
	(x_{11}^2 I_1^2 + x_{11} I_1 J_2) \cap J_2^2
	&= (x_{11}^2 I_1^2 + x_{11} I_1 J_2) \cap (x_{11}) \cap J_2^2 \\
	&= (x_{11}^2 I_1^2 + x_{11} I_1 J_2) \cap x_{11} J_2^2 \\
    &= x_{11} J_2^2.
\end{align*}
This means that the short exact sequence simplifies to
\begin{eqnarray}
0 \longrightarrow x_{11}J_2^2\longrightarrow
(x_{11}^2I_1^2+x_{11}I_1J_2)\oplus J_2^2\longrightarrow J^2 \longrightarrow
0.
\end{eqnarray}
By the induction hypothesis,
the two incarnations of $J_2^2$ and hence of $x_{11}J_2^2$
have the same Hilbert function,
so that to prove that the two incarnations $I^2$ and $(\ini(I))^2$ of $J^2$
have the same Hilbert function
it suffices to prove that the two incarnations
of $x_{11}^2I_1^2+x_{11}I_1J_2$ have the same Hilbert function.
As before we get a short exact sequence:
\begin{eqnarray} \label{ses3}
0 \longrightarrow x_{11}I_1^2\cap I_1J_2 \longrightarrow x_{11}I_1^2\oplus I_1J_2\longrightarrow x_{11}I_1^2+I_1J_2 \longrightarrow 0,
\end{eqnarray}
and since $x_{11}$ is a non-zerodivisor on $I_1J_2$
and $J_2 \subseteq I_1$,
the short exact sequence simplifies to
\begin{eqnarray} \label{ses2}
0 \longrightarrow x_{11} I_1 J_2 \longrightarrow x_{11}I_1^2\oplus I_1J_2\longrightarrow x_{11}I_1^2+I_1J_2 \longrightarrow 0.
\end{eqnarray}

The conclusion
will follow once we show that in the two incarnations of $J_2$ the Hilbert function of $I_1 J_2$ is the same, that is $\ini(I_1I_2)=I_1\ini(I_2)$.
Clearly, $I_1\ini(I_2) \subseteq \ini(I_1I_2)$.
Define a grading on $R$ by setting $\deg(x_{1i})= 0$ and $\deg(x_{2j}) = 1$ for all $i, j$.
With this grading, $I_1$ is generated by elements of degree $1$,
$I_2$ is generated by elements of degree $1$,
and $I_1I_2$ is generated by elements of degree $2$.
It suffices to show that for every homogeneous $f \in I_1I_2$,
$\ini(f) \in I_1 \ini(I_2)$.
Clearly $\ini(f)$ has degree at least $2$ and is in $\ini(I_2)$.
By Proposition~\ref{gb of ini},
$\ini(I_2)$ is generated by elements of degree $1$.
Thus by degree reasons,
to write $\ini(f)$ as an element of $\ini(I_2)$,
the coefficient must have degree at least $1$,
i.e.,
the coefficient must be in
$(x_{2i} \mid i = r+1, \ldots, n) = I_1$.
It follows that $\ini(f) \in I_1 \ini(I_2)$.
\end{proof}

We are now ready to prove the main theorem of this section. Let $A$ be a subalgebra of a polynomial ring over a field $K$.
We recall that a subset $C$ of $A$ is a \emph{SAGBI basis} for $A$ with respect to a term order $\tau$ if $\ini(A)$ is generated as a $K$-algebra by the initial monomials of the elements in $C$ with respect to $\tau$. In general, SAGBI bases are difficult to compute and may not even be finite. When $I$ is the ideal of maximal minors of a generic $m \times n$ matrix $X$, Conca showed that $\init(\Ri)=\Rini$ with respect to a diagonal term order $\tau$~\cite[Theorem~2.1]{ConcaGB}.
Moreover, Narasimhan \cite[Corollary~3.4]{Nar} showed that the maximal minors of $X$ form a  Gr\"{o}bner basis for $I$. It then follows that the natural algebra generators of $\Ri$ are a SAGBI basis of it. The following theorem extends these results to the case of the ideal maximal minors of a sparse $2\times n$ matrix.

\begin{theorem}\label{initial rees thm}
Let $X$ be a sparse $2\times n$ matrix and let $I$ be the ideal of $2\times 2$ minors of $X$. Let $\tau$ be a diagonal term order on $R=K[X]$ and let $\tau'$ be its extension to $R[t]$ as in~(\ref{tau'}). Then $$\init(\Ri)=\Rini \ \ \mbox{ and } \ \ {\rm{reltype}}(I)\le 2.$$ 
Moreover, the set $\{x_{1i}, x_{2j} \mid 1\le i \le r+s, r+1 \le j \le n\} \cup\{f_{ij}t\mid 1\le i<j \le n\}$ is a SAGBI basis of $\Ri$ with respect to $\tau'$. 
\end{theorem}

\begin{proof}
According to \cite[Proposition~5.1]{CN}, the defining ideal of the special fiber ring of~$I_{\lambda} = \ini(I)$ is the ideal generated by the $2 \times 2$ minors of $L_{\lambda}$. By~\cite[Theorem~3.1]{V} we know that $\Rini$ is of fiber type and hence the relation type of $\ini(I)$ is at most $2$. 

By Lemma~\ref{I^2=} we have that $\ini(I^2)=(\ini(I))^2$, and thus by Theorem~\ref{reltype thm}, the relation type of $I$ is at most~$2$ and $\ini(I^i)=(\ini(I))^{i}$ for all $i$.
Therefore,
$$\init(\Ri)=\oplus_{i\ge 0} \ini(I^i)=\oplus_{i\ge 0} (\ini(I))^i=\Rini.$$

The last part of the claim now follows, since the $2\times 2$ minors of $X$ are a Gr\"{o}bner basis of $I$ by \cite[Proposition~5.4]{Boocher}. Therefore the set $\{x_{1i}, x_{2j} \mid 1\le i \le r+s, r+1 \le j \le n\} \cup\{f_{ij}t\mid 1\le i<j \le n\}$ is a SAGBI basis of $\Ri$, since the leading terms of these elements generate $\Rini$ as $K$-algebra and $\init(\Ri)=\Rini$. 
\end{proof}

An immediate consequence of the equality $\ini(I^i)=(\ini(I))^i$ is the following corollary.

\begin{corollary}
Let $X$ be a sparse $2\times n$ matrix and let $I$ the ideal of $2\times 2$ minors of $X$. Let $\tau$ be a diagonal term order on $R=K[X]$. The $i$-fold products of maximal minors of $X$ are a $\tau$-Gr\"{o}bner basis of $I^i$  for every $i$. 
\end{corollary}

Using the theory of SAGBI bases, one can now deduce properties of the Rees algebra of $I$ from those of its  initial algebra. 

\begin{corollary}\label{rees sing}
Let $X$ be a sparse $2\times n$ matrix and let $I$ the ideal of $2\times 2$ minors of $X$. Then $\Ri$ has rational singularities if the field $K$ has characteristic $0$ and it is $F$-rational if $K$ has positive characteristic. In particular, $\Ri$ is a Cohen-Macaulay normal domain.
\end{corollary}

\begin{proof}
Let $\tau$ be a diagonal term order on $R=K[X]$ and let $\tau'$ be its extension to $R[t]$ as in~(\ref{tau'}).
By \cite[Corollary 5.3 and Theorem 5.9]{SVV}, $\Rini$ is a Cohen-Macaulay normal domain. Thus $\init(\Ri)$ is a Cohen-Macaulay normal domain by Theorem~\ref{initial rees thm}.
The conclusion then follows from \cite[Corollary~2.3]{CHV}.
\end{proof}

\section{The Defining Equations of the Rees Algebra}

In this section we obtain the defining equations of the Rees algebra of the ideal of $2\times 2$ minors of a sparse $2\times n$ matrix $X$. In addition to the setup and notation from Section~\ref{sectsetting}, we will also adopt the following.

\begin{setting}\label{setting 2}
Let $S=K[L_{\lambda}']$ and $T=K[L_{\lambda}]$ be the polynomial rings over $K$ in the variables that appear in $L_{\lambda}'$ and $L_{\lambda}$, respectively. 
To simplify our notation, we make the identifications:\begin{itemize} 
\item $x_{ui}=0$ if $u=1$ and $i>r+s$ or  if $u=2$ and $i\le r$, 
\item $y_{ij}=0$ if $i,j\in \{1,\ldots,r\}$ or $i,j \in \{r+s+1,\ldots, n\}$, 
\end{itemize}
to give meaning to all other $x_{ui}, y_{ij}$ with $u \in \{1,2\}$, $1 \le i< j \le n$. 
Define the following standard  presentations
of the symmetric algebra $\Sym$,
the Rees algebra $\Ri$ of $I$,
and the special fiber ring $\Fi$:
\begin{eqnarray*}
\sigma: S \longrightarrow \Sym,\\
\rho:S\longrightarrow \Ri,\\
\varphi:T\longrightarrow \Fi,
\end{eqnarray*}
where for all $i,j$,
$\sigma(x_{ij})=\rho(x_{ij})=x_{ij}$, $\sigma(y_{ij})=\varphi(y_{ij})=f_{ij}$, and $\rho(y_{ij})=f_{ij}t$. 
We let $\mathcal{L}=\ker(\sigma)$, $\mathcal{J}=\ker(\rho)$, and $\mathcal{K}=\ker (\varphi)$. The ideals $\mathcal{L}, \mathcal{J}, \mathcal{K}$ are called the defining ideals of the symmetric algebra, the Rees algebra, and the special fiber ring of $I$, respectively. 

We similarly define the presentations of the symmetric algebra,
of the Rees algebra, and of the special fiber ring of $\ini(I)$:
\begin{eqnarray*}
\sigma': S &\longrightarrow& \Symin,\\
\rho': S &\longrightarrow&\Rini= \init(\Ri), \\
\varphi':T&\longrightarrow& \Fini,\\
\end{eqnarray*}
where for all $i,j$,
$\sigma(x_{ij})=\rho(x_{ij})=x_{ij}$, $\sigma(y_{ij})=\varphi(y_{ij})=\ini(f_{ij})$, and $\rho(y_{ij})=\ini(f_{ij})t$. 
Moreover, let $\mathcal{L}'=\ker(\sigma')$, $\mathcal{J}'=\ker(\rho')$, and $\mathcal{K'}=\ker (\varphi')$ denote the defining ideals of the symmetric algebra, the Rees algebra, and the special fiber of $\ini(I)$, respectively. 
\end{setting}

\begin{remark} \label{definition ell, plucker}
Let $\ell_{uijk}=x_{ui}y_{jk}-x_{uj}y_{ik}+x_{uk}y_{ij}$ and $p_{ijkl}=y_{ij}y_{kl}-y_{ik}y_{jl}+y_{il}y_{jk}$
for $u \in\{1, 2\}$ and $1\le i<j<k<l \le n$. 
With the identifications from the Setting and Notation~\ref{setting 2}, $\ell_{uijk}, p_{ijkl}$ are in $S$ for all $u\in\{1,2\}$, $1\le i<j<k<l\le n$. We call the $\ell_{uijk}$'s the linear relations of $I$ and the $p_{ijkl}$'s the Pl\"{u}cker relations of $I$. Clearly  $\ell_{uijk} \in \mathcal{L}$ and $\ell_{uijk}, p_{ijkl} \in \mathcal{J}$ for all $u \in\{1,2\}$ and $1\le i<j<k<l \le n$.
\end{remark}

In the next proposition we describe the defining equations of the symmetric algebra, the special fiber, and the Rees algebra  of $\ini(I)$. In addition, we show that the $\ell_{uijk}$'s are indeed the defining equations of the symmetric algebra of $I$.

\begin{proposition} \label{L L' K' J'}
Adopt Setting and Notation~\ref{setting 2}. Then
\begin{enumerate}[$($a$)$]
\item $\mathcal{L} =(\ell_{uijk}\mid u = 1, 2; 1\le i <j <k\le n)$.
\item $\mathcal{L} '$ is generated by the $2 \times 2$ minors of $L_{\lambda}'$ that involve either the first column or the first row.
\item $\mathcal{K}' =I_2(L_{\lambda})$.
\item $\mathcal{J}'  =I_2(L_{\lambda}')=\mathcal{L}' +I_2(L_{\lambda})S$.
\end{enumerate}
Moreover, the natural generators of $I_2(L_{\lambda})$ and $I_2(L_{\lambda}')$ are Gr\"{o}bner bases of
$\mathcal{K}'$ and $\mathcal{J}'$, respectively, with respect to a diagonal term order.
\end{proposition}

\begin{proof}
(a) By \cite[Theorem~4.1 and its proof]{Boocher} the presentation matrix of $I$ is obtained from the Eagon-Northcott resolution by ``pruning''.
In particular,  the relations on the generators~$f_{ij}$ of~$I$ arise from taking the $3 \times 3$ minors of
the $3 \times n$ matrix obtained from $X$ by doubling one of the two rows.
These yield precisely the relations on the symmetric algebra of the given form.

Item (b) follows from \cite[Theorem~5.1]{HHV} since Ferrers ideals satisfy the $\ell$-exchange property, see e.g.~\cite[Lemma~6.3]{LS}. 

(c) This follows from \cite[Proposition~5.1]{CN}.

(d) Since $\ini(I)$ is a Ferrers ideal, it is the edge ideal of a bipartite graph. Therefore the Rees algebra of $\ini(I)$ is of fiber type by \cite[Theorem~3.1]{V}.

The last part follows from the fact that $I_2(L_{\lambda})$ and $I_2(L_{\lambda}')$ are ladder determinantal ideals. Therefore, by~\cite[Corollary~3.4]{Nar}, the $2\times 2$ minors of $L_{\lambda}$ and $L_{\lambda}'$ are Gr\"{o}bner bases of $I_2(L_{\lambda})$ and $I_2(L_{\lambda}')$, respectively, with respect to a diagonal term order.
\end{proof}

The description of the defining equations of the Rees algebra of $\ini(I)$, in combination with our result from the previous section that shows that $\init(\Ri)=\Rini$, allows us to deduce that the Rees algebra of $I$ is a Koszul algebra. 

\begin{corollary}\label{Rees koszul}
Let $X$ be a sparse $2\times n$ matrix, $I$ the ideal of $2\times 2$ minors of $X$. Then $\Ri$ is a Koszul algebra and $I$ has linear powers. In particular, ${\rm{reg}} (I^k)=2k$ for all $k \in \mathbb{N}$.
\end{corollary}

\begin{proof}
The defining ideal of $\Rini$ has a Gr\"{o}bner basis of quadrics by Proposition~\ref{L L' K' J'}, so $\Rini$ is a Koszul algebra, see~\cite{Fro} and~\cite[Theorem~6.7]{EH}. By Theorem~\ref{initial rees thm} we have $\ini(\Ri)=\Rini$ and since $\Rini$ is a Koszul algebra, then $\Ri$ is a Koszul algebra, by \cite[Corollary~2.6]{CHV}.
According to Blum~\cite[Corollary~3.6]{Bl}, if the Rees algebra of an ideal is Koszul, then the ideal has linear powers, i.e., the powers of the ideal have a linear resolution. 
\end{proof}

We now are ready to prove the main result of this article.  

\begin{theorem}\label{gens of rees}
Let $X$ be a sparse $2\times n$ matrix, $I$ the ideal of $2\times 2$ minors of $X$. Adopt Setting and Notation~\ref{setting 2}.
\begin{enumerate}[$($a$)$]
\item The defining ideal $\mathcal{J}$  of $\Ri$
is generated by the linear relations $\ell_{uijk}$ for $u \in\{1, 2\}$ and $1\le i<j<k\le n$, and the Pl\"{u}cker relations $p_{ijkl}$ for $1\le i<j<k<l \le n$. Moreover, these generators form a Gr\"obner basis of $\mathcal{J}$ with respect to a suitable weight.
\item The Rees algebra $\Ri$ is of fiber type, i.e., $$\mathcal{J}=\mathcal{L}+\mathcal{K}S.$$
\item The Pl\"ucker relations $p_{ijkl}$ for $1\le i<j<k<l\le n$ are the defining equations of $\Fi$.
\end{enumerate}
\end{theorem}

\begin{proof} 
(a) We let $G$ be the set of all $\ell_{uijk}$ and $p_{ijkl}$ as in Remark~\ref{definition ell, plucker}. 
We claim that $\mathcal{J}=(G)$. Define a weight $\omega$ on $R[t]$ and a weight $\pi$ on $S$ as follows:
$\omega(x_{1j})= 1,$ $\omega(x_{2j})= j,$ $\omega(t)=1,$
$\pi(x_{1j})= 1,$ $\pi(x_{2j})= j,$ $\pi(y_{ij})=\omega(\ini(f_{ij})t)=j+2.$

By Proposition~\ref{L L' K' J'}, $I_2(L_{\lambda}')$ is the defining ideal of $\Rini= \init(\Ri)$.
The weight $\omega$ represents the term order $\tau$, that is, $\inid_{\omega}(f_{ij})=\ini(f_{ij})$. Therefore, by~\cite[Proposition in Lecture 3.1]{Stur} (see also \cite[Theorem~11.4]{Stur2}) we conclude that $I_2(L_\lambda')=\inid_\pi(\mathcal{J})$.

Since $G \subseteq \mathcal{J}$, to prove that $G$ is a Gr\"obner basis of $\mathcal{J}$ with respect to $\pi$, it suffices to prove that each $2\times 2$ minor of $L_\lambda'$ is the leading form with respect to $\pi$ of some element in $G$. This also implies that $\mathcal{J}$ is generated by $G$.

We first analyze the $2 \times 2$ minors of $L'_\lambda$ that involve the first column.
These are of the form $E_{1ijk} = x_{1i} y_{jk} - x_{1j} y_{ik}$ with $i<j<k$ and they are homogeneous of weight $k+3$. Since $x_{1k}y_{ij}$ has weight $j+3$, $E_{1ijk}$ is the leading form of $\ell_{1ijk} \in G$.

We next analyze the $2 \times 2$ minors of $L'_\lambda$ that involve the first row.
These are of the form $E_{2ijk} = x_{2k} y_{ij} - x_{2j} y_{ik}$ with $i<j<k$ and they are homogeneous of weight $j+k+2$. Since $x_{2i}y_{jk}$ has weight $i+k+2$, $E_{2ijk}$ is the leading form of $\ell_{2jik} \in G$.

It remains to prove that each $2 \times 2$ minor of $L_\lambda$ is a unit multiple of a leading form of an element in $G$. Such a minor is of the form $F_{ijkl} = y_{il} y_{jk} - y_{ik} y_{jl}$ for some $i<j<k<l$ and it is homogeneous of weight $k+l+4$. Since $y_{ij}y_{kl}$ has weight $j+l+4$, $F_{ijkl}$ is the leading form of $p_{ijkl} \in G$.

Items (b) and (c) follow immediately from (a).
\end{proof}

The following corollary is a direct consequence of Theorem~\ref{gens of rees}.

\begin{corollary} \label{ini F}
Let $X$ be a sparse $2\times n$ matrix, $I$ the ideal of $2\times 2$ minors of $X$, and $\tau$ a diagonal term order on $R=K[X]$. Then $\ini(\Fi)=\Fini$. Moreover, the $2\times 2$ minors of $X$ are a SAGBI basis of the $K$-algebra $\Fi$, the Pl\"ucker relations form a Gr\"obner basis  for $\mathcal{K}$, and $\Fi$ is a Koszul algebra. 
\end{corollary}

\begin{proof}
By Proposition~\ref{L L' K' J'}, the defining equations of $\Fini$ are the $2 \times 2$ minors of $L_\lambda$, i.e.,  $\Fini=K[\ini(f_{ij})]\cong T/I_2(L_\lambda)$. For any $i<j<k<l$, let $F_{ijkl}=y_{il} y_{jk} - y_{ik} y_{jl}$ be a generator of $I_2(L_\lambda)$.
In the proof of Theorem~\ref{gens of rees} we showed that $p_{ijkl}=F_{ijkl}+y_{ij}y_{kl}$. Hence $\varphi(F_{ijkl})=-\varphi (y_{ij}y_{kl})$ and $\ini \varphi(F_{ijkl})= \ini \varphi (y_{ij}y_{kl}) $. Then by \cite[Proposition~1.1]{CHV} the generators of $I$ are a SAGBI basis of $\Fi$, in particular $\ini(\Fi)=\Fini$. Moreover, the Pl\"ucker relations form a Gr\"obner basis of $\mathcal{K}$ by~\cite[Corollary~11.6]{Stur2}, hence $\Fi$ is a Koszul algebra.
\end{proof}

Finally, we obtain the analogous result for $\Fi$ as in Corollary~\ref{rees sing}.

\begin{corollary}\label{FI CM normal}
Let $X$ be a sparse $2\times n$ matrix and $I$ the ideal of $2\times 2$ minors of $X$. Then $\Fi$ has rational singularities if the field $K$ has characteristic $0$ and is $F$-rational if $K$ has positive characteristic. In particular, $\Fi$ is a Cohen-Macaulay normal domain.
\end{corollary}

\begin{proof}
By \cite[Theorem 7.1]{SVV} and Corollary~\ref{rees sing},
$\Fini$ is normal.
Thus, by \cite[Theorem 1]{Ho}
$\Fini$ is Cohen-Macaulay as well.
By Corollary~\ref{ini F}, we have $\ini(\Fi)=\Fini$.
The result now follows from \cite[Corollary~2.3]{CHV}.
\end{proof}

\begin{corollary}\label{invFi}
Let $X$ be a sparse $2\times n$ matrix and $I$ the ideal of $2\times 2$ minors of $X$.  Then
\begin{enumerate}[$($a$)$]
\item $\Fi$ has dimension $\min\{n+s-1,2n-r-2\}$.
\item $\Fi$ is Gorenstein if $r\le 2$.
\item ${\rm{reg}}(\Fi)=\min\{n-3,n-r-1,r+s-1\}$. 
\item The $a$-invariant of $\Fi$ is $a(\Fi)=\min\{-r-s,r-n,-s-2\}$, unless $n=r+s$ in which case $a(\Fi)=-n$ if $r=1$ or $a(\Fi)=-n+1$ otherwise.
\end{enumerate}
\end{corollary}

\begin{proof}
(a) The formula follows from \cite[Corollary~2.6]{CHV} in combination with \cite{SVV} or \cite[Section~4]{Conca}.

(b) By \cite[Proposition~2.5]{Conca} $\Fini$ is Gorenstein if and only if $r\le 2$. The statement now follows from \cite[Corollary~2.3]{CHV}.

(c) From \cite[Proposition~5.7]{CN} we obtain the regularity of $\Fini$,
 which simplifies to ${\rm{reg}}(\Fini)=\min\{n-3,n-r-1,r+s-1\}$. By  \cite[Corollary~2.5]{CHV} we have $a(\Fini)=a(\Fi)$ and since $\Fini$ is Cohen-Macaulay, then ${\rm{reg}}(\Fi)={\rm{reg}}(\Fini)$.

(d) The result follows from the fact that $a(\Fi)=-\dim \Fi +{\rm{reg}}(\Fi)$. Alternatively, the formula is a special case of \cite[Corollary~9]{GK}. In the special case $n=r+s$, we have $\dim \Fi=2n-r-2$ and ${\rm{reg}}(\Fi)=\min\{n-3, n-r-1\}$. 
\end{proof}

\begin{corollary}\label{eFI}
Let $X$ be a sparse $2\times n$ matrix and $I$ the ideal of $2\times 2$ minors of $X$. Let $\lambda$ be the partition for the Ferrers ideal $\ini(I_{\lambda})$.
Then the normalized Hilbert series of $\Fi$ is $p_{\lambda}(z)=1+h_1(\lambda)z+\ldots +h_{r+s-1}(\lambda)z^{r+s-1}$, where $h_k(\lambda)=\binom{r+s-1}{k}\binom{n-r-1}{k}-\binom{s+1}{k+1}\binom{n-3}{k-1}$ with the convention that $\binom{j}{i}=0$ if $j<i$.
Moreover, the multiplicity of $\Fi$ is $e(\Fi)=\binom{n+s-2}{n-r-1}-\binom{n+s-2}{n-1}$.
\end{corollary}

\begin{proof} First notice that the $h$-vectors for $\Fi$ and $\Fini$ coincide and in particular, $e(\Fi)=e(\Fini)$ by \cite[Corollary~2.5]{CHV}. Since $\Fini$ is a ladder determinantal ring the formula for the Hilbert series is obtained in \cite[Theorem~4.7]{Wang}. One can then deduce the formula for the multiplicity immediately. The formula for the multiplicity is also worked out explicitly in \cite[Corollary~4.2]{Wang1}. 
\end{proof}

\begin{corollary}
Let $X$ be a sparse $2\times n$ matrix and $I$ the ideal of $2\times 2$ minors of $X$.  Then
\begin{enumerate}[$($a$)$]
\item The regularity of $\Ri$ is ${\rm{reg}}(\Ri)=\min\{n-1,n-r,r+s\}$. 
\item The $a$-invariant of $\Ri$ is $a(\Ri)=\min\{-s-2,-s-r-1,r-n-1\}$.
\item The normalized Hilbert series of $\Ri$ is $p_{\lambda}(z)=1+h_1(\lambda)z+\ldots +h_{r+s}(\lambda)z^{r+s}$, where $h_1(\lambda)=(r+s)(n-r)-\binom{s+1}{2}-1$ and $h_k(\lambda)=\binom{r+s}{k}\binom{n-r}{k}-\binom{s+1}{k+1}\binom{n-1}{k-1}$ for $k\neq 1$, with the convention that $\binom{j}{i}=0$ if $j<i$.
\item The multiplicity of $\Ri$ is $e(\Ri)=\binom{n+s}{n-r}-\binom{n+s}{n+1}-1$.
\end{enumerate}
\end{corollary}

\proof
Let \begin{align*}
\lambda&=(\underbrace{n-r, \ldots, n-r}_{r \ {\rm{ times }}}, n-r-1, n-r-2, \ldots, n-r-s), \\
\mu&=(\underbrace{n-r+1, \ldots, n-r+1}_{r+1 \ {\rm{ times }}}, n-r, n-r-1, \ldots, n-r-s+1),
\end{align*}
and let $I_2(L'_\lambda)$ be the defining ideal of $\Rini$. By~\cite[Theorem~2.1]{Gorla}, the ideal $I_2(L_\mu)$ is obtained from the ideal $I_1(L_\lambda)$ generated by the entries of $L_\lambda$ by an ascending G-biliaison of height 1 on $I_2(L'_\lambda)$. Specifically, one has 
\begin{equation}\label{biliaison}
yI_2(L_\mu)+I_2(L'_\lambda)=fI_1(L_\lambda)+I_2(L'_\lambda),
\end{equation}
where $y$ is the variable appearing in position $(2,2)$ of $L_\mu$ and $f$ is the $2\times 2$-minor of the first two rows and columns of $L_\mu$.

(a) By \cite[Corollary~2.5]{CHV} we have $a(\Rini)=a(\Ri)$ and since $\Rini$ and $\Ri$ are Cohen-Macaulay, then ${\rm{reg}}(\Ri)={\rm{reg}}(\Rini)$. Since ${\rm{reg}}(K[L_\mu]/I_1(L_\lambda))=0$, then $${\rm{reg}}(\Ri)={\rm{reg}}(K[L_\mu]/I_2(L_\mu))=\min\{n-1,n-r,r+s\},$$ where the first equality follows from~\cite[Theorem~3.1]{DG} and the second from Corollary~\ref{invFi}. 

(b) The formula for the $a$-invariant of $\Ri$ now follows from the fact that $a(\Ri)=-\dim \Ri +{\rm{reg}}(\Ri)=\min\{-s-2,-s-r-1,r-n-1\}$. 

We now prove (c) and (d) together. The normalized Hilbert series of $\Ri$ and $\Rini$ coincide and in particular $e(\Ri)=e(\Rini)$ by \cite[Corollary~2.5]{CHV}. Using the short exact sequences
$$0\longrightarrow I_2(L'_\lambda)(-1) \longrightarrow I_2(L_\mu)(-1) \oplus I_2(L'_\lambda) \longrightarrow yI_2(L_\mu)+I_2(L'_\lambda)\longrightarrow 0$$
and
$$0\longrightarrow I_2(L'_\lambda)(-2) \longrightarrow I_1(L_\lambda)(-2) \oplus I_2(L'_\lambda) \longrightarrow fI_1(L_\lambda)+I_2(L'_\lambda)\longrightarrow 0$$ together with (\ref{biliaison}), 
one obtains $$HS_{K[L_\mu]/I_2(L_\mu)}(z)=zHS_{K[L_\mu]/I_1(L_\lambda)}(z)+(1-z)HS_{K[L_\mu]/I_2(L'_\lambda)}(z),$$
where $HS_A(z)$ denotes the Hilbert series of the algebra $A$.
Therefore, by Corollary~\ref{eFI} $$h_1(\lambda)=h_1(\mu)-1=(r+s)(n-r)-\binom{s+1}{2}-1$$ 
and 
$$h_k(\lambda)=h_k(\mu)=\binom{r+s}{k}\binom{n-r}{k}-\binom{s+1}{k+1}\binom{n-1}{k-1} 
\ \mbox{ for }\ k\neq 1,$$ where $1+h_1(\mu)z+\ldots+h_{r+s}(\mu)z^{r+s}$ is the normalized Hilbert series of $K[L_\mu]/I_2(L_\mu)$.
In particular, the multiplicity of $\Rini$ is $$e(\Rini)=e(K[L_\mu]/I_2(L_\mu))-e(K[L_\mu]/J)=\binom{n+s}{n-r}-\binom{n+s}{n+1}-1. \eqed$$

We close with the following remark that gives an alternative path for a proof of our results for special types of $m\times n$ sparse matrices. 

\begin{remark}\label{ASL}
Let $X$ be an $m\times n$ sparse matrix, $m\leq n$. We assume that, after row and column permutations, the variables that appear in $X$ form a two-sided ladder as in Figure~\ref{fig:stair}.

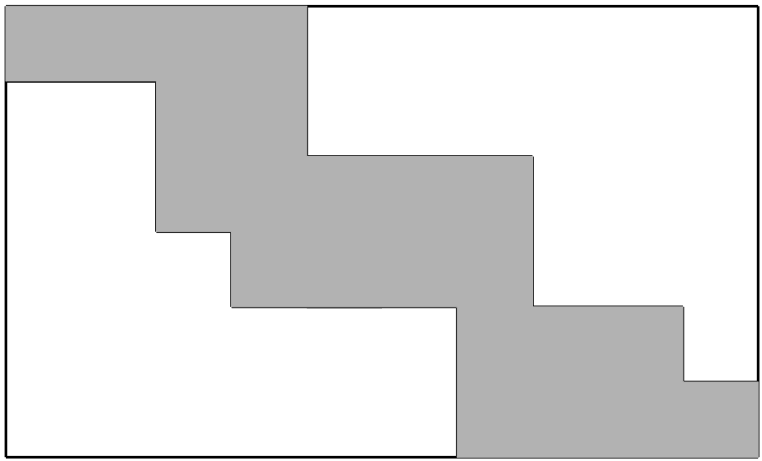
\begin{figure}[h]
	\centering
	\begin{tikzpicture}[ every node/.style={scale=0.9}]
	\draw [line width=1pt, color=black] (0,5)--(10,5);
	\draw [line width=1pt, color=black] (0,-1)--(10,-1);
	\draw [line width=1pt, color=black] (0,-1)--(0,5);
	\draw [line width=1pt, color=black] (10,-1)--(10,5);
	\draw [line width=1pt, color=black] (0,4)--(2,4);
		\draw [line width=1pt, color=black] (4,1)--(6,1);
	\draw [line width=1pt, color=black] (4,3)--(4,5);
	\draw [line width=1pt, color=black] (2,2)--(3,2);
	\draw [line width=1pt, color=black] (2,4)--(2,2);
	\draw [line width=1pt, color=black] (7,1)--(7,3);
		\draw [line width=1pt, color=black] (3,1)--(5,1);
	\draw [line width=1pt, color=black] (3,1)--(3,2);
	\draw [line width=1pt, color=black] (9,1)--(9,0);
	\draw [line width=1pt, color=black] (4,3)--(7,3);
	\draw [line width=1pt, color=black] (7,1)--(9,1);
	\draw [line width=1pt, color=black] (6,1)--(6,-1);
		\draw [line width=1pt, color=black] (9,0)--(10,0);
			\filldraw [gray!60](0,4)--(0,5)--(4,5)--(4,3)--(7,3)--(7,1)--(9,1)--(9,0)--(10,0)--(10,-1)--(6,-1)--(6,1)--(3,1)--(3,2)--(2,2)--(2,3)--(2,4);
	\end{tikzpicture} 
	\caption{A matrix with a shaded two-sided ladder.}
	\label{fig:stair}
\end{figure}

The case of a sparse $2\times n$ matrix $X$ is a special case of the above type of matrix with $m=2$, where the ladder has one lower inside corner in position $(1,r+1)$ and one upper inside corner in position $(2,r+s)$. Let $I=I_m(X)$ and let $\tau$ be a diagonal term order. Notice that an $m\times m$-minor of $X$ is non-zero if and only if all the entries on its diagonal are non-zero. Moreover, the $m\times m$-minors of $X$ form a universal Gr\"obner basis of $I_m(X)$ by~\cite[Proposition~5.4]{Boocher}. It follows that $\ini(I)$ is generated by the products of the elements on the diagonals of the $m\times m$ non-zero minors of $X$.

One can show that $\ini(I)$ is a sortable ideal with respect to the lexicographic order induced by the following order of the variables: $x_{i,j}>x_{k,l}$ if either $i<k$ or $i=k$ and $j<l$. Therefore, the defining ideal of $\Fini$ is the toric ideal generated by the binomial relations obtained by the sorting \cite[Theorem~6.16]{EH}. Moreover, these generators are a quadratic Gr\"{o}bner basis of the defining ideal of $\Fini$. In particular,  the sorting we use allows us to conclude that $\ini(I)$ is a generalized Hibi ideal and  $\Fini$ is a generalized Hibi ring, see \cite{HH05} and \cite[Section 6.3]{EH}. 
Therefore, $\Fini$ is a normal, Cohen-Macaulay, Koszul algebra.

Furthermore,  $\ini(I)$ is a weakly polymatroidal ideal with respect to the same order on the variables as above. For the definition of a weakly polymatroidal ideal, see \cite{KH} or \cite[Definition~6.25]{EH}.  Hence, by \cite[Proposition~6.26]{EH} $\ini(I)$ satisfies the $\ell$-exchange property with respect to the sorting order. Therefore, $\Rini$ is of fiber type and the defining equations of $\Rini$ are precisely the defining equations of the special fiber ring $\Fini$ and the linear relations by \cite[Theorem~6.24]{EH}.

Because of the special shape of our matrix, one can use the criterion in \cite{RS} (see also~\cite[Proposition~1.1]{CHV}) and proceed in a similar manner as in \cite[Theorem~6.46]{EH} to show that the Pl\"{u}cker relations are the defining equations of $\Fi$ and the maximal minors of $X$ are a SAGBI basis of $\Fi$. 
One can also show that Pl\"{u}cker relations along with the linear relations of $\mathcal{S}(I)$ are the defining equations of $\Ri$ and the maximal minors of $X$ along with the variables of $R$ are a SAGBI basis of $\Ri$. In particular $\init(\Ri)=\Rini$ and $\ini(\Fi)=\Fini$, where $\tau^\prime$ is a suitable term order. One then obtains similar results for $\Ri$ and $\Fi$ as in Corollaries~\ref{rees sing},~\ref{Rees koszul},~\ref{ini F}, and ~\ref{FI CM normal}. 
\end{remark}

\acknowledgement{We are grateful to the Women in Commutative Algebra (WICA) group for organizing the first ``Women in Commutative Algebra" workshop at BIRS, Banff, Canada, where this project began. We also thank the staff at BIRS for their great hospitality during the workshop. The WICA workshop was funded by the National Science Foundation grant DMS 1934391 and by the Association for Women in Mathematics grant NSF-HRD 1500481.}

\end{document}